\documentclass{singlecol-new}

\usepackage{natbib,stfloats}
\usepackage{mathrsfs}
\usepackage{graphicx}

\theoremstyle{TH}{
\newtheorem{lemma}{Lemma}
\newtheorem{theorem}[lemma]{Theorem}

}

\theoremstyle{THrm}{

}

\theoremstyle{THhit}{

}

\makeatletter
\def\theequation{\arabic{equation}}

\def\BottomCatch{%
\vskip -10pt
\thispagestyle{empty}%
\begin{table}[b]%
\NINE\begin{tabular*}{\textwidth}{@{\extracolsep{\fill}}lcr@{}}%
\\[-12pt]
\end{tabular*}%
\vskip -30pt%
\end{table}%
} \makeatother

\begin{document}%

\setcounter{page}{1}

\LRH{A. Bencheikh et~al.}

\RRH{Operational matrix for fractional Emden-Fowler problem. }

\VOL{x}

\ISSUE{x}

\PUBYEAR{xxxx}

\BottomCatch

\CLline


\subtitle{}

\title{\sf{\textbf{A new operational matrix based on Boubaker polynomials for solving fractional Emden-Fowler problem}}}

\authorA{\sf{Abdelkrim Bencheikh, Lakhdar Chiter *}}

\affA{Department of Mathematics, Ouargla University,   Algeria, \\
Department of Mathematics, Setif University, Algeria,\\
Fundamental and Numerical Mathematics Laboratory
\\
E-mail: krimbench@yahoo.fr\\
E-mail: benchiekh.abdkrim@univ-ouargla.dz \\
E-mail: lakhdarchiter61@gmail.com\\
E-mail: lchiter@univ-setif.dz\\
{\sf{*}}Corresponding author}

\begin{abstract}
In this paper the singular Emden-Fowler equation of fractional order is
introduced and a computational method is proposed for its numerical solution. 
For the approximation of the solutions we have used Boubaker polynomials and defined the formulation for its fractional derivative operational matrix. 
This tool was not used yet, however, this area has not found many practical applications yet, and here introduced for the first time. 
The operational matrix of the Caputo fractional derivative tool converts these problems to a system of algebraic equations whose solutions are simple and easy to compute. 
Numerical examples are examined to prove the validity and the effectiveness of the proposed method to find approximate and precise solutions. 
\end{abstract}

\KEYWORD{Boubaker Polynomials; Operational matrix of fractional derivatives;
Collocation method; Fractional Emden-Fowler Type Equations.}

\REF{to this paper should be made as follows: Bencheikh, A., Chiter, L. (xxxx) ` A new operational matrix based on Boubaker polynomials for solving fractional Emden-Fowler problem 
pp.xxx--xxx.}

\begin{bio}
Bencheikh Abdelkrim received his Phd in Applied Mathematics from the University of Setif, Setif, Algeria. Currently, he is working as an Associate Professor in the department of Mathematics, Ouargla University, Algeria. 
His research interests are numerical methods of integral and integro-differential equations and related fields.\vs{8}
\noindent 
Chiter Lakhdar is a professor at the department of Mathematics at the University of Setif, Algeria. His research interests are optimization and numerical methods.\vs{8}

\end{bio}

\maketitle

\section{Introduction}

In mathematical physics and nonlinear mechanics there exists sufficiently large number of particular
basic singular fractional differential equations for which an exact analytic solution in terms of
known functions did not exist ( \cite{1, 2, 3, 4, 5} ). One of these equations describing  many phenomena in mathematical physics and astrophysics such as,  the thermal behaviour of a spherical,
cloud of gas, isothermal gas sphere and theory of stellar structure, theory of thermionic currents among many others, is called the singular Emden-Fowler equations of fractional order formulated as: \cite{S1,S2,Hu,Jo}
\begin{equation}\label{0.1}
D^{2\alpha }u(x)+\frac{\lambda}{x^{\alpha }}D^{\alpha }u(x)+s(x)g(u( x))
=h(x),
x\in(0,1),
\lambda>0,\frac{1}{2}<\alpha\leq 1 
\end{equation}
subject to the conditions:
\begin{equation*}\label{0.2}
u\left( 0\right) =a,\quad D^{\alpha}u\left( 0\right) =b,
\end{equation*}
where $ a $ and $ b $ are constants. When $ \alpha = 1, \lambda = 2 $, and $ h(x) = 1 $, Eq.\eqref{0.1} becomes the Lane-Emden type equation.
$ D^{\alpha } $ denote the Caputo fractional derivatives. It is generally defined as follows :
\begin{equation}\label{0.3}
D^{\alpha }u\left( x\right) =\frac{1}{\Gamma \left( n-\alpha \right) }%
\int_{0}^{x}\frac{u^{\left( n\right) }\left( t\right) }{\left( x-t\right)
^{\alpha -n+1}}dt,\quad  n-1<\alpha <n,n\in 
\mathbb{N},\alpha > 0
\end{equation}
For the Caputo derivative we have $ D^{\alpha }C =0, $
where $ C $ is a constant, and
\begin{eqnarray*}\label{0.4}
D^{\alpha }x^{\beta } &=&\left\{ 
\begin{array}{c}
0, \\ 
\frac{\Gamma \left( \beta +1\right) }{\Gamma \left( \beta +1-\alpha \right) }%
x^{\beta -\alpha }%
\end{array}%
\begin{array}{c} 
\text{for }  \beta \in 
\mathbb{N}
\cup \left\{ 0\right\}   \text{and }  \beta <\left\lceil \alpha \right\rceil \\ 
\text{for } \beta \in 
\mathbb{N}
\cup \left\{ 0\right\}  \text{ and } \beta \geq \left\lceil \alpha
\right\rceil  \text{ or } \beta \notin 
\mathbb{N}
_{\text{ }}\text{and }\beta >\left\lceil \alpha \right\rceil%
\end{array}%
\right.
\end{eqnarray*}
where, $ \lceil \alpha \rceil $ denotes the integer part of $ \alpha $, that is the largest integer less or equal than 
$ \alpha $, or
the smallest integer greater than or equal to $ \alpha $.
\\
The problem \eqref{0.1}  was studied by using the Residual Power Series Method by \cite{S1},
 Homotopy analysis method (HAM) by \cite{Hu},
 Reproducing kernel Hilbert space method by \cite{S2}, The fractional differential transformation (FDT) \cite{Jo}, Polynomial Least Squares Method by \cite{Bo}, Shifted Legendre Operational Matrix by \cite{Lege}, Chebyshev wavelets by \cite{Kin}, Orthonormal Bernoulli's polynomials by \cite{Be},  Orthonormal Bernstein polynomials by \cite{AB}. For the solution of the classic Emden-Fowler equations (Case $ \alpha =1 $), there are many Studies
of analytical as well as numerical methods is provided in
monographs by \cite{cla1,cla2,cla3,cla4,cla5,
cla6,cla7}.\\
The purpose of this paper is to use Boubaker operational matrix of fractional order for solving a singular initial value problems of fractional Emden-Fowler type equations \eqref{0.1}. To the best of our knowledge this is the first time that the Boubaker operational matrices are used to obtain solutions of singular Emden-Fowler equations of fractional order. First we present a new theorem which can reduce the fractional Emden-Fowler problem to a system of algebraic equations.
 The Boubaker polynomials  were established for the first time by Boubaker ( $2007$ ), to solve heat
equation inside a physical model. The first monomial definition of the Boubaker polynomials on
interval $ x \in [0,1] $, was introduced by \cite{Bou1,Bou2,Bou3,Bou4,Bou5,
Bou7}:
\begin{equation}\label{0.5} 
\mathbf{B}_{0}(x)=1,\quad  
 \mathbf{B}_{n}(x)=\sum_{p=0}^{\xi(n)}%
[\frac{(n-4p)}{(n-p)}C_{n-p}^{p}] (-1)^{p}x^{n-2p},\quad n\geq 1,
\end{equation}%
where
$ \xi(n)=\lfloor{{\frac{i}{2}}}\rfloor = \frac{2n+((-1)^{n}-1)}{4}$ and $ C_{n-r}^{r}=\frac{(n-p)!}{r!(n- 2p)!} $.
The symbol $\lfloor .\rfloor $  denotes the floor function.
The Boubaker polynomials could be calculated by following recursive formula:
\begin{equation}\label{0.6}
\mathbf{B}_{m}(x)=x\mathbf{B}_{m-1}(x)-\mathbf{B}_{m-2}(x),\quad m\geq 2. 
\end{equation}
We will construct  operational matrix of Caputo fractional derivatives 
$\mathbf{D}^{(\alpha) }$ 
for the Boubaker polynomials which are given by
\begin{equation}\label{0.7}
  D^{\alpha}\textbf{B}(x)\simeq \mathbf{D}^{(\alpha)}\textbf{B}(x),
\end{equation}
where 
$ \textbf{B}(x)=[B_{0}(x) ,B_{1}(x),\ldots,B_{N}(x)]^{T}$ be Boubaker vector
and the matrice $ \mathbf{D}^{\left( \alpha \right)} $ are of order $ (N+1)\times (N+1)$.
In order to show the high performance of Boubaker operational matrix of fractional derivative, we apply it to solve 
equation \eqref{0.1}.
\\
The paper is organized as follows. In Section \eqref{2}, we express Boubaker polynomials in terms of Taylor basis, and function approximation. In Section \eqref{3} The operational matrix of Caputo fractional derivatives is constructed.
In Section\eqref{4}, we use Boubaker polynomials method for solving fractional Emden-Fowler type equations. Section\eqref{5} illustrates some numerical examples to show the accuracy of this method. Finally, Section \eqref{6}
concludes the paper.

\section{Boubaker's matrix and approximation of function}\label{2}
By using the expression \eqref{0.5} and taking $ n=0,...,N $, we can express
Boubaker polynomials in terms of Taylor basis  ( \cite{Bou4,Bou5,Bou7} )
\begin{equation}\label{1.1}
\mathbf{B}\left( x\right) =\mathbf{M}\textbf{T}\left( x\right) ,\quad x\in \left[
0,1\right] ,
\end{equation}%
where 
\begin{equation}\label{1.2}\textbf{T}(x)=[1,x,\ldots ,x^{N}]^{T},
\end{equation}%
and if $ N $ is odd,
$$
\textbf{M}=\left[ 
\begin{array}{ccccccc}
m_{0,0} & 0 & 0 & 0 & \cdots & 0 & 0 \\ 
0 & m_{1,0} & 0 & 0 & \cdots & 0 & 0 \\ 
m_{2,1} & 0 & m_{2,0} & 0 & \cdots & 0 & 0 \\ 
\vdots & \vdots & \vdots & \vdots & \ddots & \vdots & \vdots \\ 
m_{N-1,\frac{N-1}{2}} & 0 & m_{N-1,\frac{N-3}{2}} & 0 & \cdots & m_{N-1,0} & 
0 \\ 
0 & m_{n,\frac{N-1}{2}} & 0 & m_{N,\frac{N-3}{2}} & \cdots & 0 & m_{N,0}%
\end{array}%
\right] 
$$
if $ N $ is even,
$$
\textbf{M}=\left[ 
\begin{array}{ccccccc}
m_{0,0} & 0 & 0 & 0 & \cdots & 0 & 0 \\ 
0 & m_{1,0} & 0 & 0 & \cdots & 0 & 0 \\ 
m_{2,1} & 0 & m_{2,0} & 0 & \cdots & 0 & 0 \\ 
\vdots & \vdots & \vdots & \vdots & \ddots & \vdots & \vdots \\ 
0 & m_{N-1,\frac{N-2}{2}} & 0 & m_{N-1,\frac{N-4}{2}} & \cdots & m_{N-1,0} & 
0 \\ 
m_{N,\frac{N}{2}} & 0 & m_{N,\frac{N-2}{2}} & 0 & \cdots & 0 & m_{N,0}%
\end{array}%
\right] 
$$
where
\begin{equation}\label{1.3}  
 \mathbf{B}_{n}(x)=\sum_{p=0}^{\xi(n)}%
m_{n,p}x^{n-2p},\quad n = 0, 1,..., N, p = 0, 1,...,\lfloor\frac{n}{2}\rfloor,
\end{equation}%
\begin{equation}\label{1.4}  
m_{n,p}=[\frac{(n-4p)}{(n-p)}C_{n-p}^{p}] (-1)^{p}.
\end{equation}%
As can be observed, $\textbf{M} $ denotes a lower triangular matrix and  $ \vert\textbf{M}\vert =$\ $\Pi _{i=0}^{N}m_{i},_{0}=1 $\, thus it is
non-singular and  $ \textbf{M}^{-1} $ exists.
\\

We consider the set of Boubaker polynomials of $  N^{th} $ degree as 
\begin{equation}\label{1.5}
\mathbf{B}\left( x\right) =\left[ B_{0}\left( x\right) ,B_{1}\left( x\right)
,...,B_{N}\left( x\right) \right] ^{T}\subset L^{2}\left[ 0,1\right],
\end{equation}
and assume that $ \mathbf{S}_{N}=span\left\{ B_{0}\left( x\right)
,B_{1}\left( x\right) ,...,B_{N}\left( x\right) \right\} $. Because $\mathbf{%
S}_{N} $ is a finite dimensional vector space, if $ u $ is an arbitrary
function in $ L^{2}\left[ 0,1\right] $, then $ u $ has the best approximation out
of  $ \textbf{S}_{N} $ such as  $ u_{N} \in \mathbf{S}_{N} $, that is \cite{Bou6}\newline
\begin{equation}\label{1.6}
\forall y\in \mathbf{S}_{N},\left\Vert u-u_{N}\right\Vert _{2}\leq
\left\Vert u-y\right\Vert _{2}.
\end{equation}
Since  $ u_{N}\in \mathbf{S}_{N} $, there exist unique coefficients  $ c_{i}, i =
0, 1, . . ., N $  such that 
\begin{equation}\label{1.7}
u\left( x\right) \simeq u_{N}\left( x\right) =\sum_{i=0}^{n}c_{i}B_{n}\left(
x\right) =\textbf{C}^{T}.\mathbf{B}\left( x\right),
\end{equation}
where  $ \textbf{C} $ is an  $ (N+ 1)\times(1) $ vector given by $ \textbf{C} = [c_{0}, c_{1},
. . ., c_{N}]^{T} $,  $\mathbf{B(x)} $ is the vector function defined in Eq.\eqref{1.5}, and coefficients vector $ C $  can be computed by  $ \textbf{C}^{T}<\mathbf{B(x)}, 
\mathbf{B(x)}> = < u(x), \mathbf{B(x)}> $, such that 
\begin{equation}\label{1.8}
< u(x), \mathbf{B(x)}> =\int_{0}^{1}u\left( x\right) \textbf{B}^{T}(x) dx,
\end{equation}
and $ <.,.> $ denotes the standard inner product on  $ L^{2}\left[ 0,1\right] $.
Thus, by definition $ \textbf{Q} = <\textbf{B}(x), \textbf{B}(x)> $, we get 
\begin{equation}\label{1.9}
\textbf{C}^{T}=(\int_{0}^{1}u\left( x\right) \textbf{B}^{T}(x) dx)\textbf{Q}^{-1},
\end{equation}
where $ \textbf{Q} $ is the following  $ (N+1)\times (N+1) $ matrix :%
\begin{eqnarray*}
\textbf{Q} =\langle \textbf{B}(x),\textbf{B}(x)\rangle &=&\int_{0}^{1}\textbf{B}(x)\textbf{B}^{T}(x)dx
 = \textbf{M}\left(
\int_{0}^{1}\textbf{T}(x)\textbf{T}^{T}(x)dx\right) \textbf{M}^{T}\\
&=& \textbf{M}\textbf{H}\textbf{M}^{T}.
\end{eqnarray*}
Where; $ \textbf{H}=[h_{ij}]_{(N+1)\times (N+1)} $ is the well-known Hilbert matrix,
the components of which can be computed as follows:%
\begin{equation}\label{1.10}
h_{ij}=\frac{1}{i+j+1},i,j=1,...,N.
\end{equation}
\begin{theorem} \label{result1}
( \cite{Bou6,Bou7} ) Elements $ B_{0},B_{1},...,B_{N} $ of a Hilbert
space  $ L^{2}[0,1] $ constitute a linearly independent set in  $ \ L^{2}[0,1] $
if and only if
\begin{equation*}
\textbf{G}(B_{0},B_{1},...,B_{N})=\left\vert 
\begin{array}{cccc}
\left\langle B_{0},B_{0}\right\rangle  & \left\langle
B_{0},B_{1}\right\rangle  & \cdots  & \left\langle B_{0},B_{N}\right\rangle 
\\ 
\left\langle B_{1},B_{0}\right\rangle  & \left\langle
B_{1},B_{1}\right\rangle  & \cdots  & \left\langle B_{1},B_{N}\right\rangle 
\\ 
\vdots  & \vdots  & \ddots  & \vdots  \\ 
\left\langle B_{N},B_{0}\right\rangle  & \left\langle
B_{N},B_{1}\right\rangle  & \cdots  & \left\langle B_{N},B_{N}\right\rangle 
\end{array}%
\right\vert \neq 0
\end{equation*}%
produces that $ \textbf{Q} $ is symmetric and also non-singular, so $ \textbf{Q}^{-1} $ exists.
\end{theorem}
\begin{lemma}(\cite{Bou6,Bou5}) Suppose that  $  u\in C^{N+1}[0,1] $  and  $\mathbf{S}_{N}=span\left\{ B_{N}\left(
x\right) ,B_{n}\left( x\right) ,...,B_{N}\left( x\right) \right\} $
Let $ u_{0} $ be the best approximation for $ u $ out of $ \mathbf{S}_{N} $ then%
\begin{equation}
\left\vert \left\vert u(x)-u_{0}(x)\right\vert \right\vert_{L^{2}[0,1]} \leq \frac{%
Max_{x\in \lbrack 0,1]}|u^{(N+1)}(x)|}{\left( N+1\right) \sqrt{2N+3}}
\end{equation}
\end{lemma}
\begin{proof}.
 For proof, see \cite{Bou6,Bou5}.
 \end{proof}
\begin{theorem}
(\cite{Bou6,Bou5})
 Suppose that $ u\in L_{2}[0,1]$ and $ u(x) $ is approximated by 
 $ \sum_{i=0}^{N}c_{i}B_{i}(x) $, then we have 
\begin{equation}\label{1.12}
\lim_{N \rightarrow \infty }\left\vert \left\vert u\left( x\right)
-\sum_{i=0}^{N}c_{i}B_{i}(x)\right\vert \right\vert _{L_{2}[0,1]}=0
\end{equation}
\end{theorem}
\section{ The Boubaker  operational matrix of fractional derivative}\label{3}
The main objective of this section is to derive the operational matrix of Caputo fractional derivatives
based on the Boubaker polynomials
\\

For a vector $\mathbf{B(x)}$, we can approximate the operational matrices of
fractional order integration as\cite{Bou5}:
\begin{equation}\label{3.1}
D^{\alpha }\mathbf{B}\left( x\right) \simeq \mathbf{D}^{\left( \alpha
\right) }\mathbf{B}\left( x\right) 
\end{equation}
where $ \mathbf{D}^{\left( \alpha \right) } $ is the $ (N+1)\times (N+1) $ 
Caputo fractional operational matrix of integration for Boubaker polynomials. We compute $ \mathbf{D}^{\left( \alpha \right) } $ as follows:
\begin{eqnarray}\label{3.2}
D^{\alpha }\mathbf{B}\left( x\right)  &\simeq &\textbf{M}D^{\alpha }\textbf{T}\left(
x\right)  
=\textbf{M}\textbf{Z}\bar{\textbf{X}}\left( x\right) 
\end{eqnarray}
where the matrix $\textbf{Z}_{(N+1)\times (N+1)}$ is given by
\begin{equation}\label{3.3}
\textbf{Z}=\left( Z_{i,j}\right) =\left\{ 
\begin{array}{c}
\frac{\Gamma \left( j+1\right) }{\Gamma \left( j+1-\alpha \right) }, \\ 
0,%
\end{array}%
\begin{array}{c}
i=j=\lceil \alpha \rceil ,...,N \\ 
\text{otherwise,}%
\end{array}%
\right. ,
\end{equation}
and $\bar{\textbf{X}}=\left[ \bar{X}_{i+1}\right] _{(N+1)\times (1)}$, with
\begin{equation}\label{3.4}
\bar{X}_{i+1}=\left\{ 
\begin{array}{c}
x^{i-\alpha }, \\ 
0,%
\end{array}%
\begin{array}{c}
i=\lceil \alpha \rceil ,...,N \\ 
i=0,1,...,\lceil \alpha \rceil -1%
\end{array}%
\right. 
\end{equation}
 Now, $\bar{X}$ is expanded in terms of Boubaker polynomials as%
\begin{equation}\label{3.5}
\bar{X}=\textbf{E}^{T}\mathbf{B(x)}
\end{equation}%
where $ \textbf{E}=\left[ e_{0},e_{1},...,e_{m}\right] $ and $\ e_{i}=\textbf{Q}^{-1}\hat{E}_{i}
$ $\hat{E}_{i}=[\hat{e}_{i,0},\hat{e}_{i,1},...,\hat{e}_{i,m}]^{T}.$The
entries of the vector $\hat{E}_{i}$ can be calculated as
\begin{equation}\label{3.6}
\hat{e}_{i,j}=(\int_{0}^{1}x^{i-\alpha }B_{j}(x)dx)\textbf{Q}^{-1}
\end{equation}%
then we have
\begin{equation}\label{3.7}
D^{\alpha }\mathbf{B}\left( x\right) \simeq \mathbf{D}^{\left( \alpha
\right) }\mathbf{B}\left( x\right) 
\begin{array}{cc}
, & \mathbf{D}^{\left( \alpha \right) }=\textbf{M}\textbf{Z}\textbf{E}^{T}%
\end{array}%
\end{equation}
$\mathbf{D}^{\left( \alpha \right) }$is the operational matrix of the Caputo fractional derivative.
\\

\section{Solution of singular Fractional Emden-Fowler problem }\label{4}
his section presents the derivation of the method for solving a singular initial value
problems of fractional Emden Fowler
type equations. \\

Let us consider the fractional Emden-Fowler
equation of the form
\begin{equation}\label{4.1}
D^{2\alpha }u(x)+\frac{\lambda}{x^{\alpha }}D^{\alpha }u(x)+s(x)g(u( x))
=h(x),\quad
x\in(0,1),\quad
\lambda>0,\quad \frac{1}{2}<\alpha \leq 1 
\end{equation}
with initial conditions: 
\begin{equation}\label{4.01}
u\left( 0\right) =a,\quad D^{\alpha}u\left( 0\right) =b
\end{equation}
Approximating $u(x)$, $s\left( x\right) g\left( u\left( x\right) \right) $
by the Boubaker polynomials  as
\begin{eqnarray}\label{4.2}
&&u\left( x\right) =\sum_{i=0}^{m}c_{i}B_{i}\left( x\right) =\textbf{C}^{T}\mathbf{B}%
\left( x\right),
\quad  s\left( x\right) g\left( u\left( x\right) \right) =s\left( x\right)
g\left( \textbf{C}^{T}\mathbf{B}\left( x\right) \right) 
\end{eqnarray}%
where the unknowns are, $\textbf{ C}=\left[ c_{0},c_{1},...,c_{m}\right] ^{T} $.
Using operational matrix of fractional derivative, Eq. \eqref{3.7} can be written as
\begin{equation}\label{4.2}
\textbf{C}^{T}\mathbf{D}^{\left( 2\alpha \right) }\mathbf{B}\left( x\right) +\frac{%
\lambda }{x^{\alpha }}\textbf{C}^{T}\mathbf{D}^{\left( \alpha \right) }\mathbf{B}%
\left( x\right) +s\left( x\right) g\left( \textbf{C}^{T}\mathbf{B}\left( x\right)
\right) =h\left( x\right) 
\end{equation}
Collocating Eq.\eqref{4.2} at $m-1$ collocation points
leads to
\begin{equation}\label{4.3}
\textbf{C}^{T}\mathbf{D}^{\left( 2\alpha \right) }\mathbf{B}\left( x_{i}\right) +%
\frac{\lambda }{x^{\alpha }}\textbf{C}^{T}\mathbf{D}^{\left( \alpha \right) }\mathbf{B%
}\left( x_{i}\right) +s\left( x_{i}\right) g\left( \textbf{C}^{T}\mathbf{B}\left(
x_{i}\right) \right) =h\left( x_{i}\right).
\end{equation}
A set of suitable collocation points is defined as follows:
\begin{equation}\label{4.4}
x_{i}=\frac{1}{2}\left( \cos \left( \frac{i\pi }{n}\right) +1\right)
,i=0,...,m-1
\end{equation}
In addition, the initial conditions \eqref{4.01} provide two algebraic equations as%
\begin{equation}\label{4.5}
\textbf{C}^{T}\mathbf{B}\left( 0\right) =a,\quad \textbf{C}^{T}\mathbf{D}^{\left( \alpha
\right) }\mathbf{B}\left( 0\right) =0
\end{equation}
Finally, we can compute the values for the components
of $ \textbf{C} $ by solving the system of eq. \eqref{4.3} and \eqref{4.5}.
Hence, the approximate solution for $ u(x)$ can be computed by using Eq.
\eqref{1.7}. 

\section{ Numerical examples}\label{5}
In this section, we applied the method presented in Section \eqref{4} to solve fractional Emden-Fowler Equation.We have done all
the numerical computations with a computer program \textbf{Matlab}
\\
 
\textbf{Example 1.}
We consider the following fractional Emden-Fowler equations :
\begin{equation}\label{5.1}
D^{2\alpha }u(x)+\frac{2}{x^{\alpha }}D^{\alpha }u(x)+u(x)^{n}
=0,\quad
\end{equation}
subject to the conditions:
$  u(0)=1 ,\quad D^{\alpha }u(0)=0 $
\begin{enumerate}
\item
For $ \alpha=1 $  and $ n = 0 $, Eq. \eqref{5.1} has as The exact solution for this problem is 
 $$ u(x)=1-\frac{1}{3!}x^{2} $$
By applying this method, and taking $ m=2 $, we find
\begin{eqnarray*}
\textbf{D}^{\left( 1\right) }=\left[ 
\begin{array}{ccc}
0 & 0 & 0 \\ 
1 & 0 & 0 \\ 
0 & 2 & 0%
\end{array}%
\right]  ,\allowbreak \textbf{D}^{\left( 2\right) }=\left[ 
\begin{array}{ccc}
0 & 0 & 0 \\ 
0 & 0 & 0 \\ 
2 & 0 & 0%
\end{array}%
\right] ,\allowbreak
\textbf{C}=\left[ 
\begin{array}{c}
c_{0} \\ 
c_{1} \\ 
c_{2}
\end{array}%
\right]
=\left[ 
 \begin{array}{c}
\frac{4}{3} \\ 
0 \\ 
-\frac{1}{6}
\end{array}%
\right]
\end{eqnarray*}
Hence, the solution is
\begin{eqnarray*}
u(x) =\textbf{C}^{T}\mathbf{B}\left( x\right) 
&=&\left[ \frac{4}{3},0,-\frac{1}{6}\right] \left[
\begin{array}{c}
1 \\ 
x \\ 
x^{2}+2%
\end{array}%
\right] =1-\frac{1}{3!}x^{2}
\end{eqnarray*}
which is the exact solution.
\item
For $ \alpha=1 $  and $ n = 1 $, Eq. \eqref{5.1} has as the exact solution\cite{cla7}  
$$ u\left( t\right) =\frac{sinx}{x} $$
We solved the above problem, by applying the technique described in Section \eqref{4} with $ m = 3 $, we approximate solution as\\
$$ u(x)=c_{0}B_{0}(x)+c_{1}B_{1}(x)+c_{2}B_{2}(x)+c_{3}B_{3}(x)=\textbf{C}^{T}\mathbf{B}(x) $$
Here, we have
\begin{equation*}
\mathbf{D}^{\left( 1\right) }\allowbreak =\left[ 
\begin{array}{cccc}
0 & 0 & 0 & 0 \\ 
1 & 0 & 0 & 0 \\ 
0 & 2 & 0 & 0 \\ 
-5 & 0 & 3 & 0%
\end{array}%
\right] ,\ \mathbf{D}^{\left( 2\right) }\allowbreak =\left[ 
\begin{array}{cccc}
0 & 0 & 0 & 0 \\ 
0 & 0 & 0 & 0 \\ 
2 & 0 & 0 & 0 \\ 
0 & 6 & 0 & 0%
\end{array}%
\right],\allowbreak
\textbf{C}=\left[ 
\begin{array}{c}
c_{0} \\ 
c_{1} \\ 
c_{2}\\
c_{3}
\end{array}%
\right] 
\end{equation*}
we find the following system
\begin{equation*}  
c_{0}+2c_{2}=1 ,\quad c_{1}+c_{3}=0 
\end{equation*} 
\begin{equation*}
c_{0}+\frac{41}{12}c_{1}+\frac{137}{16}c_{2}+\frac{2465}{192}c_{3}=0 ,\quad
c_{0}+\frac{33}{4}c_{1}+\frac{129}{16}c_{2}+\frac{721}{64}c_{3}=0%
\end{equation*}
which has the solution:
\begin{equation*}
\lbrace{ c_{0}=\frac{25\,673}{19\,113},\quad c_{1}=-\frac{256}{19\,113}%
,\quad c_{2}=-\frac{3280}{19\,113}  ,\quad c_{3}=\frac{256}{19\,113}} 
\end{equation*}
So, in this case the approximate of $ u_{3}$ (x) is
\begin{eqnarray*}
u_{3}\left( x\right)  &=&
\left[ 
\begin{array}{cccc}
\frac{25\,673}{19\,113} & -\frac{256}{19\,113} & -\frac{3280}{19\,113} & 
\frac{256}{19\,113}%
\end{array}%
\right] \left[ 
\begin{array}{c}
1 \\ 
x \\ 
x^{2}+2 \\ 
x^{3}+x%
\end{array}%
\right]\\
&=&\frac{256}{19\,113}x^{3}-\frac{3280}{19\,113}x^{2}+1
\end{eqnarray*}

Table.\eqref{Table 1} shows the absolute errors between the approximate solutions obtained for values of $ m=3 $, and $ m=6 $ using the operational matrix of Boubaker polynomials and the exact solutions.
It can be seen from table.\eqref{Table 1} that the solutions obtained by the present method is nearly identical with the exact solutions. Clearly, increasing more higher the values of $m$ leads to highly accurate results.
\begin{table}[h!]
\caption{ Absolute error for different values of $m$ for   $ \alpha = 1$ }\label{Table 1}
\begin{center}
\begin{tabular}{|c|ccccc|}
\hline
$x$ & $0.1$ & $0.3$ & $0.5$ & $0.7$ & $0.9$ \\ \hline
$m=3$ & $3.6881$E-$5$ & $1.5070$E-$4$ & $7.9560$E$-5$ & $1.9380$E$-4$ & $%
3.9614$E$-4$ \\ \hline
$m=6$ & $3.\,\allowbreak 389\,1$E-$8$ & $3.\,\allowbreak 051\,2$E-$7$ & $%
8.\,\allowbreak 479\,1$E-$8$ & $1.\,\allowbreak 671\,4$E-$7$ & $%
2.\,\allowbreak 961\,1\times $E-$7$ \\ \hline
\end{tabular}
	\end{center}
\end{table}
\end{enumerate}

\textbf{Example 2.} We consider the following fractional Emden-Fowler equations (\cite{S1}):
\begin{equation*}
D^{2\alpha }u\left( x\right) +\frac{1}{x^{\alpha }}D^{\alpha }u\left(
x\right) +\left( 1+x^{\alpha }\right) \left( u\left( x\right) \right)
=h\left( x\right) 
\end{equation*}
subject to the conditions:
 $$ u\left( 0\right) =3,\quad D^{\alpha }u\left( 0\right) =0 $$
where 
 \begin{equation*}
h\left( x\right) =\Gamma \left( 1+2\alpha \right) +\frac{\Gamma \left(
1+2\alpha \right) }{\Gamma \left( 1+\alpha \right) }+\left( 1+x^{\alpha
}\right) \left( 3+x^{2\alpha }\right), 
\end{equation*} 
The exact solution is
$
u\left( x\right) =3+x^{2\alpha }
$\\
Applying the method developed in Sections \eqref{3} and \eqref{4} for $ m = 2$ , $ \alpha=1 $
we have:
$$ u(x)=c_{0}B_{0}(x)+c_{1}B_{1}(x)+c_{2}B_{2}(x)=C^{T}\mathbf{B}(x) $$
Therefore using Eq. \eqref{4.3} we obtain:
$ 1.\,\allowbreak 5c_{0}+2.\,\allowbreak 75c_{1}+7.\,\allowbreak
375c_{2}=\allowbreak 8.\,\allowbreak 875\allowbreak
$\\
Now, by applying Eq.\eqref{4.3} we have:
$
c_{0}+2c_{2}=3 $ and $ c_{1}=0
$\\
Finally, we get  $ c_{0}=1,c_{1}=0 $   and $ c_{2}=1 $\\
Thus we can write
\begin{equation*}
u\left( x\right) =\left[ 
\begin{array}{ccc}
1 & 0 & 1%
\end{array}%
\right] \left[ 
\begin{array}{c}
1 \\ 
x \\ 
x^{2}+2%
\end{array}%
\right] =3+\allowbreak x^{2}
\end{equation*}
which is the exact solution.
\\
Table.\eqref{Table2} shows the absolute errors between the approximate solutions obtained for values of $ \alpha=0.7,0.8 $, and $ \alpha=1 $  and the exact solutions.
\begin{table}[h!]
\caption{ Absolute error for different values of $\alpha$ for   $ m=2 $  }\label{Table2}
\begin{center}
\begin{tabular}{|c|ccccc|}
\hline
$x$ & $0.1$ & $0.3$ & $0.5$ & $0.7$ & $0.9$ \\ \hline
$\alpha =1$ & $0.000$ & $0.000$ & $0.000$ & $0.000$ & $0.000$ \\ \hline
$\alpha =0.85$ & $\allowbreak 6.\,\allowbreak 760\,7\times 10^{-2}$ & $%
1.07\,59\times 10^{-2}$ & $\allowbreak 8.\,\allowbreak 371\,2\times 10^{-3}$
& $\allowbreak 6.\,\allowbreak 982\,7\times 10^{-3}$ & $1.18\,70\times
10^{-3}$ \\ \hline
$\alpha =0.75$ & $\allowbreak 9.\,\allowbreak 703\,2\times 10^{-2}$ & $%
1.02\,64\times 10^{-2}$ & $5.\,\allowbreak 164\,3\times 10^{-3}$ & $%
4.\,\allowbreak 989\,1\times 10^{-3}$ & $1.99\,24\times 10^{-3}$ \\ \hline
\end{tabular}
	\end{center}
\end{table}
\\
\textbf{Example 3.} We consider the following fractional Emden-Fowler equations:
\begin{equation}\label{5.00}
D^{2\alpha }u(x)+\frac{\lambda }{x^{\alpha }}D^{\alpha }u(x)+-2(2x^{2}+3)u(x)
=h(x)
\end{equation}
subject to the conditions:
$  u(0)=1 ,\quad D^{\alpha }u(0)=0 $,
with $ \alpha=1,\lambda=2 $ and $ h(x)= 0 $ .Eq \eqref{5.00} has as the exact solution\cite{cla7}
$ u\left( t\right) =exp(x^{2})$ \\

In fig\eqref{fig1}, we plotted the exact solution and the approximate
solutions of $ u(x)$ for $ m = 4 $ and $ 6 $. Definitely, by increasing
the value of $ m $, the approximate value of $ u(x)$ will close to the exact values. and fig\eqref{fig1.2} present the absolute error in this case.
\begin{figure}[htp]
\center{\caption{ Graph of exact solution and approximate solution (at $ m = 4 $ and $6 $ ) .  }\label{fig1}
 \includegraphics[width=2in]{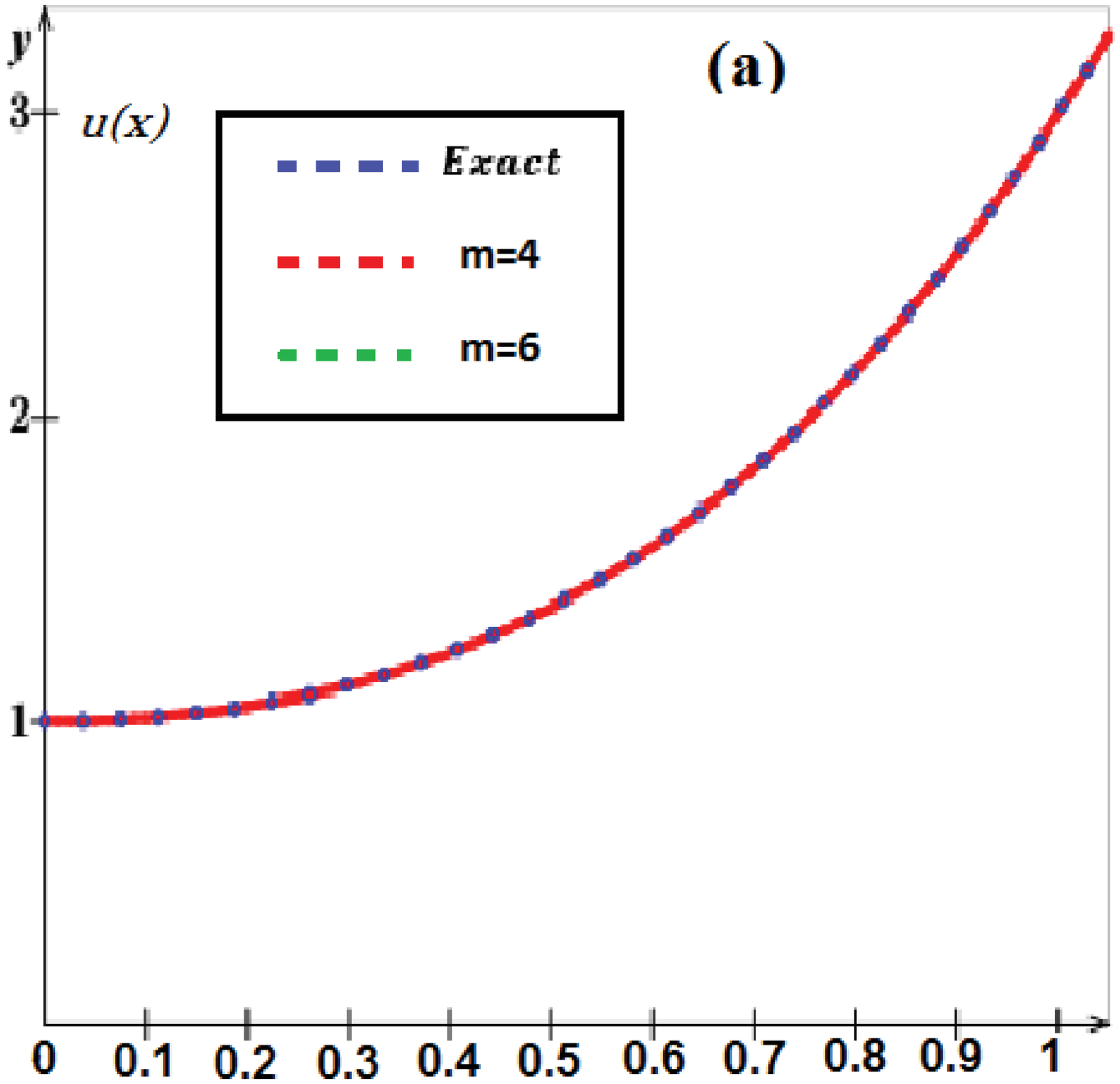}
\caption{ Graph of absolute errors.}\label{fig1.2} 
 \includegraphics[width=2in]{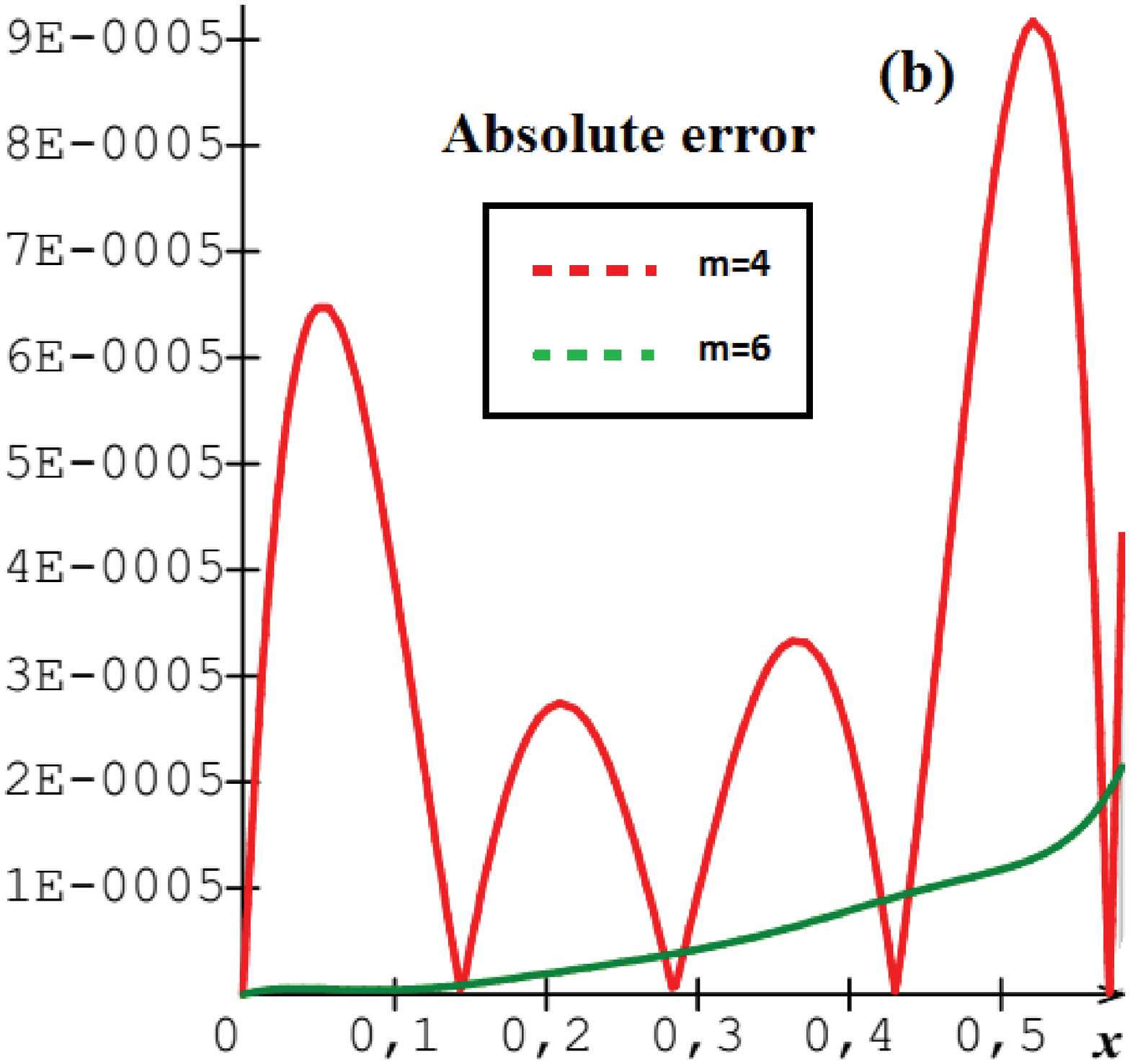}}
\end{figure}
\textbf{Example 4.}
 We consider the following fractional Emden-Fowler equations ( \cite{S1} ) :
\begin{equation*}\label{C4}
D^{2\alpha }u\left( x\right) +\frac{1}{x^{\alpha }}D^{\alpha }u\left(
x\right)- 9u\left( x\right) =h\left( x\right) ,\quad x\in (0,1),\quad \frac{1%
}{2}<\alpha \leq 1
\end{equation*}
subject to the boundary conditions
 $$ u\left( 0\right) =2,\quad D^{\alpha }u\left( 0\right) =0
 $$
where
\begin{equation*}
h\left( x\right) =-9+\frac{\Gamma \left( 1+2\alpha \right) }{\Gamma \left(
1+\alpha \right) }+\Gamma \left( 1+2\alpha \right) +\left( \frac{\Gamma
\left( 1+3\alpha \right) }{\Gamma \left( 1+\alpha \right) }+\frac{\Gamma
\left( 1+3\alpha \right) }{\Gamma \left( 1+2\alpha \right) }\right)
x^{\alpha }-9x^{2\alpha }-9x^{3\alpha }
\end{equation*}
The exact solution is
$$
u\left( x\right) =1+x^{2\alpha }+x^{3\alpha }
$$
For $ \alpha= 0.7 ,0.8, 1$ and $ m = 4 $. The operational matrix for $\mathbf{D}^{\left( \alpha\right) }$ for various  $ \alpha $
are given as 
\begin{eqnarray*}
\mathbf{D}^{\left( 0.7\right) } &=&\left[ 
\begin{array}{ccccc}
0 & 0 & 0 & 0 & 0 \\ 
7.\,\allowbreak 506\,7 & -5.\,\allowbreak 108\,4 & -7.\,\allowbreak 096\,9 & 
8.\,\allowbreak 245\,8 & -3.\,\allowbreak 488\,8 \\ 
-0.758\,54 & -2.\,\allowbreak 489 & -1.\,\allowbreak 828\,7 & 
4.\,\allowbreak 118\,2 & -2.\,\allowbreak 181\,5 \\ 
3.\,\allowbreak 851\,8 & -5.\,\allowbreak 522\,9 & -5.\,\allowbreak 082\,6 & 
8.\,\allowbreak 471\,2 & -3.\,\allowbreak 305\,7 \\ 
1.\,\allowbreak 864\,6 & -2.\,\allowbreak 331\,3 & -0.318\,72 & 
2.\,\allowbreak 375\,2 & 0.614\,68%
\end{array}%
\right]  \\
\mathbf{D}^{\left( 1.4\right) } &=&\left[ 
\begin{array}{ccccc}
0 & 0 & 0 & 0 & 0 \\ 
0 & 0 & 0 & 0 & 0 \\ 
4.\,\allowbreak 738\,6 & 7.\,\allowbreak 309\,6 & 0.564\,06 & 
-4.\,\allowbreak 312\,0 & 2.\,\allowbreak 820\,3 \\ 
-7.\,\allowbreak 585\,9 & 0.987\,11 & 3.\,\allowbreak 296\,3 & 0.211\,52 & 
-0.475\,93 \\ 
-3.\,\allowbreak 893\,4 & -10.\,\allowbreak 541 & -1.\,\allowbreak 816\,9 & 
11.\,\allowbreak 281 & -3.\,\allowbreak 742\,3%
\end{array}%
\right] 
\end{eqnarray*}
\begin{eqnarray*}
\mathbf{D}^{\left( 0.8\right) } &=&\left[ 
\begin{array}{ccccc}
0 & 0 & 0 & 0 & 0 \\ 
6.\,\allowbreak 198\,6 & -3.\,\allowbreak 393\,4 & -5.\,\allowbreak 244\,3 & 
5.\,\allowbreak 809\,3 & -2.\,\allowbreak 380\,9 \\ 
-2.\,\allowbreak 661\,6 & 3.\,\allowbreak 804\,3 & 2.\,\allowbreak 649\,3 & 
-2.\,\allowbreak 973\,3 & 1.\,\allowbreak 315\,1 \\ 
1.\,\allowbreak 755\,7 & -3.\,\allowbreak 665\,2 & -2.\,\allowbreak 863\,6 & 
5.\,\allowbreak 913\,2 & -2.\,\allowbreak 224\,9 \\ 
2.\,\allowbreak 393\,3 & -4.\,\allowbreak 664\,7 & -1.\,\allowbreak 793\,3 & 
5.\,\allowbreak 076\,8 & -0.584\,78%
\end{array}%
\right]  \\
\mathbf{D}^{\left( 1.6\right) } &=&\left[ 
\begin{array}{ccccc}
0 & 0 & 0 & 0 & 0 \\ 
0 & 0 & 0 & 0 & 0 \\ 
13.\,\allowbreak 216 & -6.\,\allowbreak 467\,1 & -11.\,\allowbreak 403 & 
12.\,\allowbreak 352 & -4.\,\allowbreak 984\,5 \\ 
-8.\,\allowbreak 582\,9 & 6.\,\allowbreak 606\,8 & 6.\,\allowbreak 319\,4 & 
-5.\,\allowbreak 044\,7 & 2.\,\allowbreak 038\,9 \\ 
-10.\,\allowbreak 082 & -5.\,\allowbreak 954\,1 & 3.\,\allowbreak 617\,9 & 
5.\,\allowbreak 905\,8 & -1.\,\allowbreak 420\,1%
\end{array}%
\right] 
\end{eqnarray*} 
\begin{eqnarray*}
\mathbf{D}^{\left( 1\right) } =\left[ 
\begin{array}{ccccc}
0 & 0 & 0 & 0 & 0 \\ 
1 & 0 & 0 & 0 & 0 \\ 
0 & 2 & 0 & 0 & 0 \\ 
-5 & 0 & 3 & 0 & 0 \\ 
0 & -4 & 0 & 4 & 0%
\end{array}%
\right],\quad  
\mathbf{D}^{\left( 2\right) } =\left[ 
\begin{array}{ccccc}
0 & 0 & 0 & 0 & 0 \\ 
0 & 0 & 0 & 0 & 0 \\ 
2 & 0 & 0 & 0 & 0 \\ 
0 & 6 & 0 & 0 & 0 \\ 
-24 & 0 & 12 & 0 & 0%
\end{array}%
\right] 
\end{eqnarray*}
Eq.\eqref{C4} with the initial condition has been solved with the proposed method and the values of the unknown matrix $ C^{T}$ are obtained and listed in Table \eqref{table2}. The values of absolute errors obtained are shown in Table \eqref{table3}.
\begin{table}[h!]
\caption{ Values of unknowns for $ m = 4 $  and for different values of $ \alpha $ }\label{table2}
\center{
\begin{tabular}{{|c|ccccc|}}
\hline
 Unknowns  & $c_{0}$ & $c_{1}$ & $c_{2}$ & $c_{3}$ & $c_{4}$ \\ \hline
$\alpha =0.7$ & $-2.\,\allowbreak 270\,6$ & $0.575\,55$ & $1.\,\allowbreak
616\,8$ & $-7.\,\allowbreak 593\,7\times 10^{-2}$ & $1.\,\allowbreak
856\,8\times 10^{-2}$ \\ \hline
$\alpha =0.8$ & $-2.\,\allowbreak 289\,5$ & $7.\,\allowbreak 948\,1\times
10^{-2}$ & $1.\,\allowbreak 635\,2$ & $0.145\,92$ & $9.\,\allowbreak
545\,0\times 10^{-3}$ \\ \hline
$\alpha =1$ & $-1.\,\allowbreak 000\,4$ & $-0.999\,65$ & $1.\,\allowbreak
000\,2$ & $0.999\,65$ & $2.\,\allowbreak 221\,6\times 10^{-5}$ \\ 
\hline
\end{tabular}}
\end{table}
\begin{table}[h!]
\caption{ Absolute errors for  $ m = 5 $ and for different values of $ \alpha $ }\label{table3}
\center{
\begin{tabular}{{|c|ccc|}}
\hline
$x$ & $\alpha =0.7$ & $\alpha =0.8$ & $\alpha =1$ \\
 \hline
$0.1$ & $\allowbreak 5.\,\allowbreak 583\,9\times 10^{-2}$ & $\allowbreak
2.\,\allowbreak 825\,1\times 10^{-2}$ & $3.\,\allowbreak 834\,8\times 10^{-5}
$ \\
 \hline
$0.2$ & $4.\,\allowbreak 923\,8\times 10^{-2}$ & $\allowbreak
2.\,\allowbreak 367\,8\times 10^{-2}$ & $\allowbreak 3.\,\allowbreak
476\,4\times 10^{-5}$ \\ \hline
$0.3$ & $\allowbreak 4.\,\allowbreak 577\,6\times 10^{-2}$ & $\allowbreak
2.\,\allowbreak 066\,5\times 10^{-2}$ & $\allowbreak 3.\,\allowbreak
1270\times 10^{-5}$ \\
 \hline
$0.4$ & $\allowbreak 4.\,\allowbreak 324\,2\times 10^{-2}$ & $\allowbreak
1.\,\allowbreak 855\,0\times 10^{-2}$ & $\allowbreak 2.\,\allowbreak
983\,1\times 10^{-5}$ \\ \hline
$0.5$ & $\allowbreak 4.\,\allowbreak 065\,4\times 10^{-2}$ & $\allowbreak
1.\,\allowbreak 719\,5\times 10^{-2}$ & $\allowbreak 3.\,\allowbreak
236\,2\times 10^{-5}$ \\ \hline
$0.6$ & $\allowbreak 3.\,\allowbreak 751\,0\times 10^{-2}$ & $\allowbreak
1.\,\allowbreak 660\,5\times 10^{-2}$ & $\allowbreak 4.\,\allowbreak
072\,1\times 10^{-5}$ \\ \hline
$0.7$ & $\allowbreak 3.\,\allowbreak 352\,8\times 10^{-2}$ & $1.6\,810\times
10^{-2}$ & $\allowbreak 5.\,\allowbreak 671\,6\times 10^{-5}$ \\
 \hline
$0.8$ & $\allowbreak 2.\,\allowbreak 853\,9\times 10^{-2}$ & $\allowbreak
1.\,\allowbreak 782\,4\times 10^{-2}$ & $\allowbreak 8.\,\allowbreak
2100\times 10^{-5}$ \\ 
\hline
$0.9$ & $\allowbreak 2.\,\allowbreak 242\,8\times 10^{-2}$ & $\allowbreak
1.\,\allowbreak 961\,9\times 10^{-2}$ & $1.\,\allowbreak 185\,7\times 10^{-4}
$ \\
 \hline
$1.0$ & $\allowbreak 1.\,\allowbreak 509\,9\times 10^{-2}$ & $\allowbreak
2.\,\allowbreak 212\,5\times 10^{-2}$ & $\allowbreak 1.\,\allowbreak
677\,8\times 10^{-4}$ \\ \hline
\end{tabular}
}
\end{table}
\begin{figure}[htp]
\center{
\caption{The graph of $u(x)$ with $ m = 4 $ and
$\alpha=  0.7, 0.8, 1 $ .}
\includegraphics[width=2in]{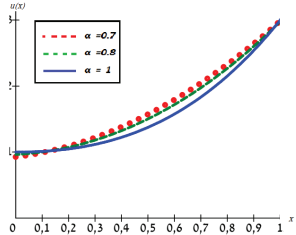}
\label{fig2}
}
\end{figure}

\section{ Conclusions}\label{6}
In this work, we get operational matrices of the fractional derivative by
Boubaker polynomials. Then by using these matrices, we reduced the singular
fractional Emden-Fowler type equations to a system of algebraic equations that can be solved easily. A Numerical example
is included to demonstrate the validity and applicability of this method, and the results reveal that the method is very effective, straightforward, simple, and it can be applied by
developing for the other related fractional problems, such as
such partial fractional differential and integro-differential
equations. This will be investigated in a future work.

\section*{Acknowledgments} We would like to thank you for \textbf{following
the instructions above} very closely in advance. It will definitely
save us lot of time and expedite the process of your paper's
publication.





\end{document}